\documentclass[a4, 12pt]{amsart}
\usepackage{amssymb}
\usepackage{amstext}
\usepackage{amsmath}
\usepackage{amscd}
\usepackage{latexsym}
\usepackage{amsfonts}
\usepackage{color}
\usepackage{enumerate}
\usepackage{textcomp}
\usepackage[all]{xy}
\usepackage{graphicx}
\usepackage{appendix}
\usepackage[]{algorithm2e}

\usepackage{mathtools}
\DeclareFontFamily{OT1}{rsfs}{}
\DeclareFontShape{OT1}{rsfs}{n}{it}{<-> rsfs10}{}
\DeclareMathAlphabet{\mathscr}{OT1}{rsfs}{n}{it}

\usepackage[pagebackref]{hyperref}
\usepackage{hyperref}
\hypersetup{
    colorlinks=true, 
    linkcolor=red, 
}

\usepackage[alphabetic,initials]{amsrefs}

\theoremstyle{plain}
\newtheorem{thm}{Theorem}[section]
\newtheorem*{thm*}{Theorem}
\newtheorem*{cor*}{Corollary}
\newtheorem*{defn*}{Definition}

\newtheorem{prop}[thm]{Proposition}
\newtheorem{lem}[thm]{Lemma}
\newtheorem{cor}[thm]{Corollary}
\newtheorem{claim}[thm]{Claim}
\newtheorem*{claim*}{Claim}

\theoremstyle{definition}
\newtheorem{defn}[thm]{Definition}

\newtheorem{ques}[thm]{Question}

\theoremstyle{remark}

\newtheorem*{tpf1}{{\sl Proof of Theorem \ref{theorem1}}}
\newtheorem*{tpf2}{{\sl Proof of Theorem \ref{theorem2}}}

\numberwithin{equation}{thm}
\def\Hom{\mathrm{Hom}}

\def\Ext{\mathrm{Ext}}

\def\Tor{\mathrm{Tor}}

\def\rank{\mathrm{rank }}
\def\a{\mathfrak a}

\def\depth{\mathrm{depth}}

\def\a_i{\underline {a_i}}

\def\coker{\mathrm{coker}}
\def\Hilb{\mathrm{Hilb}}
\def\Pic{\mathrm{Pic}}

\begin{document}

\title[]{Classification and geometric properties of surfaces with property ${\bf N}_{3,3}$}

\author[H.L. Truong]{Hoang Le Truong}
\address{Institute of Mathematics, VAST, 18 Hoang Quoc Viet Road, 10307
Hanoi, Viet Nam}
\email{hltruong@math.ac.vn\\
	truonghoangle@gmail.com}

\thanks{2020 {\em Mathematics Subject Classification\/}: 
14J60(Primary),  14D20, 14M06, 13C14 (Secondary).}
\keywords{ Adjunction mapping, Liaison,  Cubic fourfolds,  Matrix factorizations, Arithmetically Cohen-Macaulay}

\begin{abstract} 
Let $X$ be a closed subscheme of codimension $e$ in a projective space. One says that $X$ satisfies property ${\bf N}_{d,p}$, if the $i$-th syzygies of the homogeneous coordinate ring are generated by elements of degree $<d+i$ for $0\le i\le p$. The geometric and algebraic properties of smooth projective varieties satisfying property ${\bf N}_{2,e}$ are well understood, and the complete classification of these varieties is a classical result. The aim of this paper is to study the next case: projective surfaces in $\Bbb P^5$ satisfying property ${\bf N}_{3,3}$. In particular, we give a classification of such varieties using adjunction mappings and  we also provide illuminating examples of our results via calculations done with Macaulay 2.  As corollaries,   we study the CI-biliaison equivalence class of smooth projective surfaces of degree $10$ satisfying  property ${\bf N}_{3,3}$ on a cubic fourfold.


\end{abstract}

\maketitle

\section{Introduction }

A classical problem in algebraic geometry and commutative algebra is the study of
the equations defining projective varieties and of their syzygies. In \cite{EGHP06}, D. Eisenbud et al. introduced the condition ${\bf N}_{d,p}$ for a graded ring as an indication of the presence of simple syzygies. Let us recall the definition. 
Let $R=k[x_0, \ldots, x_n]$ denote the homogeneous coordinate ring of the projective space $\Bbb P^n$ over an algebraically closed field $k$ of characteristic zero, and let $I_X\subset R$ denote the homogeneous ideal of an algebraic set $X\subset \Bbb P^n$ of codimension $e$. 
We say that $X $ satisfies property  ${\bf N}_{d,p}$($p\le \infty$), for some $d \ge 2$, if 
$$\beta_{i,j}(X):=\dim \Tor^R_i(R/I_X, k)_{i+j} = 0,$$ for $i \le p$ and $j \ge d$. This is the same as property ${\bf N}_p$ defined by Green and Lazarsfeld in \cite{GrL85} if $X$ is projectively normal and $d = 2$. Notice that property ${\bf N}_{d,\infty}$ means that $X$ is $d$-regular.  One says that $X$ is arithmetically Cohen–Macaulay (ACM) with a $d$-linear resolution if the only non-zero Betti numbers are $\beta_{0,0}(X)$ and $\beta_{i,d-1}(X)$ for $i =1, 2, \ldots, e$.

One of the main interesting problems is to understand geometric properties of algebraic sets
satisfying property ${\bf N}_{d,p}$. 
In \cite{AhK15}, J. Ahn and S. Kwak showed that if $X$ satisfies property  ${\bf N}_{d,e}$, then 
$$\deg(X) \le \binom{e+d-1}{d-1}.$$
The set of algebraic sets satisfying the above equality is quite large. It contains many known examples such as the algebraic set defined by the ideal of maximal minors of a $1$-generic $d \times (e +d -1)$-matrix of linear forms, or the $(d -1)$-secant variety of a rational normal curve of degree $(e +2d -3)$.  

For $d =2$, $X$ satisfies property ${\bf N}_{2,e}$ if and only if $X$ is  ACM with $2$-linear resolution (see \cite[Corollary1.8]{EGHP05} and  \cite[Corollary1.11]{EiG84}). It follows that if in addition $\deg(X) =1 +e$, then 
$X$ is a variety of minimal degree. 
Recall that a variety of minimal degree is either a rational normal scroll, a quadric
hypersurface $Q_n \subset \Bbb P^{n+1}$, or the  Veronese surface in $ \Bbb P^5$ (see \cite{EiH87}). 
In this direction, the papers \cite{AlR02, AhK11, BeE10, BEL91,  EGHP05, GrL88, Kwa98, KwP05, Laz87} give a small sample of the kinds of investigations that have been carried out. However, not very much is actually known about algebraic sets satisfying property ${\bf N}_{d,p}$, for $d \ge3$. 
For $d =3$, J. Ahn and S. Kwak showed in \cite{AhK15} that 
$X$ satisfies property ${\bf N}_{3,e}$ and $\deg(X) =\binom{e+2}{2}$ if and only if $X$ is ACM with a $3$-linear resolution.
Along this line, it is natural to study the classification problem of smooth projective varieties $X$ satisfying  property  ${\bf N}_{3,e}$ and $\deg(X) =\binom{e+2}{2}$. For $e=2$, smooth projective varieties $X$ of degree $6$ satisfying  property  ${\bf N}_{3,2}$ are defined by the ideal of maximal minors of a $1$-generic $3 \times 4$-matrix of linear forms. Thus the first step would be a classification of smooth projective varieties $X$  of degree $10$ and codimension $3$ satisfying  property  ${\bf N}_{3,3}$. 
For curves, R. Hartshorne gave a complete classification of projective complex curves of degree $10$ satisfying  property  ${\bf N}_{3,3}$ in $\Bbb P^4$ (see \cite{Har01}). 
For surfaces, a numerical classification of rational surfaces of degree $10$ was given by Livorni in \cite{Liv90}. But  the problem of the existence of such surfaces remains open.  In particular, 
it is still an open problem to give a complete geometric classification or characterization of projective complex surfaces of degree $10$ satisfying  property  ${\bf N}_{3,3}$
(\cite[Section 6]{Man00}, \cite{Man01} and \cite[Question 3]{AHK20}).  The following theorem which is the main result of this paper gives an  answer to this problem.

\begin{thm}\label{theorem1}
Let $X$ be a non-degenerate projective surface in $\Bbb P^5$. Then  the statements are equivalent. 
\begin{itemize}
\item[$1)$] $\deg X=10$ and $X$ satisfies property  ${\bf N}_{3,3}$.
\item[$2)$] $X$  is ACM with the $3$-linear resolution
	
	\begin{equation}\label{eqn}
		\begin{tabular}{ c | c c c c c c}
			&   & $0$ & $1 $ &$2 $  &$3 $\\ 
			\hline
			$0$&  & $1    $ & $\cdot$ & $\cdot$ & $\cdot$ \\ 
			$1$&  & $\cdot$ & $\cdot$ &$\cdot$ &$\cdot$ \\ 
			$2$&  & $\cdot$ & $10   $ &$  15 $ &$  6  $ \\ 
		\end{tabular}
		\tag{$*$}
		\end{equation}
	\item[$3)$] $X$ is  an ACM surface of degree $10$ and sectional genus $6$.
\end{itemize}
In this case, $X$ is smooth and $X$ belongs to one of the seven families as in Table \ref{table:S106}.

\begin{table}[h!]
\centering
		\begin{tabular}{ |r| r| c| c|c|c|c|c|c|}
	\hline
	  $K^2_X$&$\mathcal{C}_{\delta}$ &  $H$ &Abstract structure &$h^0(\mathcal N_{X/Y})$\\ 
	\hline
	$-6$&$\mathcal{C}_8$   & $X=\Bbb P^2(5;1^{15})$  & Veronese surface   &$12$	\\ 
	\hline
	$-5$&$\mathcal{C}_{14}$   & $X=\Bbb P^2(6;2^4,1^{10})$  & del Pezzo surface of degree $5$ 	&$10$\\ 
	\hline
	$-4$&$\mathcal{C}_{20}$   & $X=\Bbb P^2(7;3^1,2^6,1^{6})$  & conic bundle over $\Bbb P^1$ of degree $6$ &$8$  	\\ 
	\hline

	$-3$&$\mathcal{C}_{26}$   & $X=\Bbb P^2(7;2^9,1^{3})$  & $\Bbb P^2$  	&$6$\\ 
	\hline
	$-2$&$\mathcal{C}_{32}$   & $X=\Bbb P^2(9;3^6,2^4,1^{1})$  & del Pezzo surface of degree $3$.  &$4$	\\ 
	\hline
	$-1$&$\mathcal{C}_{38}$   & $X=\Bbb P^2(10;3^{10})$  & $\Bbb P^2$    &$2$	\\ 
	\hline
	$0$&$\mathcal{C}_{44}$   & $X$ is a Fano model of  &  	&$0$\\ 
	&  &   an Enriques surface  &  	&\\ 
	\hline
	
\end{tabular}
\caption{ Smooth projective surfaces $X\subset Y\subset\Bbb P^5$ of degree $10$  with property  ${\bf N}_{3,3}$, where $Y$ is a cubic fourfolds and  $\mathcal{C}_{\delta}$ is a locus   of special cubic fourfolds  $Y$ of a discriminant $\delta$ (see Section {\ref{Cdelta}}). }
\label{table:S106}
\end{table}

\end{thm}

Let us explain how this paper is organized. This paper is divided into
5 sections. We give a proof on our main result and consider some applications and examples
in Section 3. In our argument, adjunction theory play an important role. Thus in Section 2 let us briefly note the adjunction theory. In Section 4 we illustrate briefly a connection between the study of smooth projective surfaces of degree $10$  with property  ${\bf N}_{3,3}$ and special cubic fourfolds.
In Section 5, as corollaries to the main theorem,   we study the CI-biliaison equivalence class of smooth projective surfaces of degree $10$ satisfying  property ${\bf N}_{3,3}$ on a cubic fourfold.

\section{preliminaries}

In this section, we will give the basic definitions and results that are needed. For a smooth surface $X$ in a projective space we denote by $H_X$ and $K_X$ the class of a hyperplane and the canonical divisor, respectively. The structure sheaf of $X$ is $\mathcal O_X$, and the Euler-Poincare characteristic of its structure sheaf is $\chi_X = \chi(\mathcal O_X)=1-q+p_g$, where $q$ is irregularity and $p_g$ is geometric genus of $X$. The degree of $X$ is $d=H^2$ and the sectional genus $\pi=\frac{1}{2}H(H+K)+1$ is the genus of its general hyperplane section. The following lemma is a special case of the Riemann-Roch Theorem.
\begin{prop}
Let $X\subseteq \Bbb P^5$ be a smooth surface of degree $d$, sectional genus $\pi$, irregularity $q$ and geometric genus $p_g$. Then we have 
$$\chi(\mathcal I_X(m))=\chi(\mathcal O_{\Bbb P^5}(m))-\binom{m+1}{2}d+m(\pi-1)-1+q-p_g.$$
\end{prop}

In fact, we proceed in a similar fashion as in \cite[Section 8 and 9]{DeS98} or \cite{TrY20a} and give a complete geometric classification or characterization of projective complex surfaces of degree $10$ satisfying  property  ${\bf N}_{3,3}$ via  the adjunction process. Thus we recall a result of Sommese and Van de Ven for a surface over $\Bbb C$.

\begin{thm}[{\cite{Som79,Van79, SoV87}, see also \cite[Theorem 8.1]{DeS98}}]\label{SV}
	Let $X$ be a smooth surface in $\Bbb P^n$ over $\Bbb C$. 
	Then the  adjoint  linear  system $|H+K|$ is non-special of dimension $N := h^0(\mathcal O_X(H+K))-1=\chi+\pi-2$ and it defines a birational morphism
	$$\Phi = \Phi_{|H+K|}: X \to \Bbb P^{N}$$
	onto a smooth surface $X_1$, which blows down precisely all $(-1)$--curves $E$ on $X$, unless
	\begin{enumerate}[(i)]
		\item  $X$ is a plane, or Veronese surface, or $Y$ is ruled by lines;
		\item  $X$ is a Del Pezzo surface or a conic bundle;
		\item $X$ belongs to one of the following four families:
		\begin{enumerate}[(a)]
			\item  	$X = \Bbb P^2(p_1,\ldots, p_7)$ embedded by $H = 6L -2(p_1+\ldots+p_7)$
			\item  	$X = \Bbb P^2(p_1,\ldots, p_8)$ embedded by $H = 6L -2(p_1+\ldots+p_7)-p_8$
			\item  	$X = \Bbb P^2(p_1,\ldots, p_8)$ embedded by $H = 9L -3(p_1+\ldots+p_8)$
			
			\item  $X = \Bbb P(\mathcal{E})$, where $\mathcal{E}$ is an indecomposable rank $2$ bundle over an elliptic curve and $H = B$, where $B$ is a section $B^2 = 1$ on $X$.

		\end{enumerate}
		
	\end{enumerate}
\end{thm}

\begin{prop}[{\cite[Proposition 8.3]{DeS98}}]\label{DS}
	Let $X$ be a surface over a field of arbitrary characteristic. Suppose
	that the adjoint linear system $|H + K |$ is base point free and that the image $X_1 \subseteq \Bbb P^{N}$
	under the adjunction map $\Phi = \Phi_{|H+K|}$ is a surface of expected degree $(H+K)^2$, expected sectional genus $\pi_1 = \frac{1}{2}(H + K)\cdot(H + 2K) + 1$, and with $\chi(\mathcal{O}_X) = \chi(\mathcal{O}_{X_1})$. Then $X_1$ is
	smooth and $\Phi : X \to X_1$ is a simultaneous blow down of  ${K_1}^2 - K^2$ many $(-1)$--lines on $X$, where $K_1$ is the canonical divisor of $X_1$.  
\end{prop}

We also need  the Hodge index theorem.
\begin{thm}\cite[Lemma 2.5.2]{BeS95} 
Let $X\subseteq \Bbb P^n$ be an irreducible surface. 
Then the intersection matrix
$$\begin{pmatrix}
H^2 & H\cdot K\\
H\cdot K & K^2
\end{pmatrix} $$
is indefinite, unless $H$ and $K$ are numerical dependent.
\end{thm}

To apply Theorem \ref{SV} repeatedly, we have to control the numerical data of the first adjoint surface $X_1\subseteq \Bbb P^N$. We have
$$H_1^2=(H+K)^2 \text{ and } H_1\cdot K_1=(H+K)\cdot K.$$
Then $K_1^2=K^2+a$, where $a$ is the number of exceptional lines on $X$, i.e. of curves $E\subseteq X$ with $E\cdot K=E^2=-1$ and $E\cdot H=1$.
Moreover, $X_1$ is a smooth projective  surface.

\begin{prop}\label{P2.5}
	Let $X = \Bbb P^2(p_1,\ldots, p_\ell)$ be the blow--up of $\Bbb P^2$ in $\ell$ points in general position. We denote by $E_1,\ldots, E_\ell$ the corresponding exceptional divisors and by $L$ the pullback of a general line in $\Bbb P^2$
	 to $X$. If $|b_0L - \sum\limits_{i=1}^{\ell} b_i E_i|$ is a very ample linear system of dimension $5$, then $X \subseteq \Bbb P^5_{\Bbb C}$ is a rational surface of degree
	$$d = b_0^2 -\sum\limits_{i=1}^{\ell}b_i^2$$
	and sectional genus
	$$\pi=\binom{b_0-1}{2}-\sum\limits_{i=1}^{\ell}\binom{b_i}{2}.$$
	The self--intersection of the canonical divisor of $X$ is $K^2 = 9 - \ell$.
\end{prop}
\begin{proof}
	Set $H=b_0L - \sum\limits_{i=1}^{\ell} b_i E_i$. Note that $L^2 = 1$, $LE_i = 0$ and $E_i^2 = -1$ for all $i=1,\ldots,\ell$.  Then we have
	$d=H^2=b_0^2 -\sum\limits_{i=1}^{\ell}b_i^2.$
	The canonical divisor of $X$ is $-3L+\sum\limits_{i=1}^{l}E_i$, so $K^2=9-\ell$. The sectional genus of $X$ can be calculated by adjunction:
    $$\pi= \frac{H(H+K)}{2}+1=\binom{b_0-1}{2}-\sum\limits_{i=1}^{\ell}\binom{b_i}{2}.$$
\end{proof}

\section{classification via adjunction mapping}
 In this section, we prove the first main result, Theorem \ref{theorem1}, using adjunction mappings.
Let $X$ be a non-degenerate  projective surface in $\Bbb P^{5}$. Suppose that  $X$ satisfies property ${\bf N}_{3,3}$ and $\deg(X)=10$.  Then by Theorem 1.2 in \cite{AhK15}, $X$ is arithmetically Cohen-Macaulay and $I_X$ has $3$-linear resolution. Moreover, $X$  has the minimal free resolution as  in \ref{eqn}
and sectional genus of $X$ is $6$. 

In general, by Lemma 1.1, Proposition 1.3 and Remarks in \cite[]{Man01} (see also \cite{Man00}),  we  know that, if  $Z$ is a non-degenerate projective Cohen-Macaulay surface in $\Bbb P^{n_0}$, then 
$$n_0 = d-\pi+1+p_a-h^2(\mathcal O_Z(1))=d-\pi+1+h^1(\mathcal O_Z(1)),$$
where $p_a$ is the arithmetic genus of $Z$, and so we have the bounds
$$n_0-1\le d\le \binom{n_0}{2}+h^1(\mathcal O_Z(1)).$$
 Moreover, if  $\deg(Z)=\binom{n_0}{2}$ and  $h^1(\mathcal O_Z(1))=0$ then the homogeneous ideal of $Z$ can always be generated by forms of degree $3$ (\cite{Man00}).
So, in $\Bbb P^5$,  projective Cohen-Macaulay surfaces $Z$ of degree $10$ sectional genus $6$ have $h^1(\mathcal O_Z(1))=0$ and so $p_a=q=p_g=0$. In particular, $X$ has $h^1(\mathcal O_X(1))=0$, and $p_a=q=p_g=0$.
\begin{lem}\label{lem1}
Let $X$ be a non-degenerate  projective surface in $\Bbb P^{5}$. Suppose that  $X$ satisfies property ${\bf N}_{3,3}$ and $\deg(X)=10$. Then we have
$$-6\le K_X^2\le 0.$$ 
\end{lem}
\begin{proof}
By the hypothesis,  $X$ is a  surface of degree $10$ and sectional genus $6$ in $\Bbb P^5$. Then we have
$HK=2.6-2-10=0  $. By the Hodge index theorem, we obtained that $K^2\le \frac{1}{10}(H\cdot K)^2=0$. 
 The adjunction mapping is defined by the linear system $H + K$. By Riemann-Roch and the Kodaira Vanishing Theorem, we get that
$$h^0(\mathcal O_X(H+K))=\chi+\pi-1=6.$$
By Theorem \ref{SV}, the adjunction map
	$$\Phi = \Phi_{|H+K|}: X \to X_1\subseteq\Bbb P^{5}$$
is birational onto a smooth surface $X_1$.
Therefore since a nondegenerate surface $Z \subset \Bbb P^N$ has degree $\deg Z \ge N - 1$ (see \cite{Xam81}), we have $$4=5-1\le \deg X_1 =H_1^2=(H+K)^2=H^2+K^2.$$
Since $H^2=10$, we have $K^2\ge -6$.
\end{proof}

\subsection{\bf Surface $X$ with $K^2=-6$} 
The intersection matrix of $X$ is 
$$\begin{pmatrix}
H^2 & H\cdot K\\
H\cdot K & K^2
\end{pmatrix} =
\begin{pmatrix}
10 & 0\\
0 & -6
\end{pmatrix}. 
$$
By Theorem \ref{SV}, we have $h^0(\mathcal O_X(H+K))=\chi+\pi-1=6$. Moreover, the adjunction map
	$$\Phi = \Phi_{|H+K|}: X \to X_1\subseteq\Bbb P^{5}$$
is birational onto a smooth surface $X_1$ with intersection matrix 
$$\begin{pmatrix}
H_1^2 & H_1\cdot K_1\\
H_1\cdot K_1 & K_1^2
\end{pmatrix} =
\begin{pmatrix}
4 & -6\\
-6 & K_1^2
\end{pmatrix}. 
$$
By the Hodge index theorem, we have $K_1^2\le 9$. Since $(H_1+K_1)^2\ge 0$, we have
$K_1^2\ge -2H_1\cdot K_1-H_1^2=8$.  It follows from the classification  of surfaces of degree $4$ (see \cite{Ion82} or \cite{Liv90}) that we have two cases.

\begin{enumerate}[{(i)}] \rm
\item {\it If $K^2_1=9$} then $X_1$ is a smooth Veronese surface with intersection matrix 
$$\begin{pmatrix}
H_1^2 & H_1\cdot K_1\\
H_1\cdot K_1 & K_1^2
\end{pmatrix} =
\begin{pmatrix}
4 & -6\\
-6 & 9
\end{pmatrix}. 
$$
Thus $\Phi : X \to X_1$ is a simultaneous blow down of  $15=K_1^2-K^2$ many $(-1)$--lines on $X$.
Since $K^2=-6$, the number of exceptional divisors of surface $X$ is $9-K^2=15$.
It follows that $X=\Bbb P^2(p_1,\ldots, p_{15})$ is a blowup 
of $\Bbb P^2$ in fifteen points in general position by the complete  linear system
$$H=(5,1^{15})=5L-\sum\limits_{i=1}^{15}E_i.$$ 
Moreover, $X$ is smooth by Proposition \ref{DS} and $ \chi_{top} = 3+15=18$.

\item {\it If $K^2_1=8$} then $X_1$ is a surface $\Bbb P^1\times \Bbb P^1\subseteq \Bbb P^5$ of degree $4$ with intersection matrix 
$$\begin{pmatrix}
H_1^2 & H_1\cdot K_1\\
H_1\cdot K_1 & K_1^2
\end{pmatrix} =
\begin{pmatrix}
4 & -6\\
-6 & 8
\end{pmatrix}. 
$$
It follows that $X=\Bbb F_0(p_1,\ldots, p_{14})$ is a blowup of Hirzebruch surface $\Bbb F_0$ in fourteen points in general position by the complete  linear system
$$H=(4,3;1^{14})=4C_0+3f-\sum\limits_{i=1}^{14}E_i,$$  
where $C_0$ stands for a section of minimal self-intersection and $f$ for a fiber.
Since $X$ has maximum degree $10=\binom{5}{2}$, to show that the surface $X$ does not satisfy property ${\bf N}_{3,3}$ 
it is sufficient to verify that $h^0(\mathcal I_X(2))=1$.  We verified this via Macaulay2 and we can conclude that $h^0(\mathcal I_X(2))=1$ (see \cite{Tr2019}).
\end{enumerate}

\subsection{\bf Surface with $K^2=-5$} 
The intersection matrix of $X$ is 
$$\begin{pmatrix}
H^2 & H\cdot K\\
H\cdot K & K^2
\end{pmatrix} =
\begin{pmatrix}
10 & 0\\
0 & -5
\end{pmatrix}. 
$$
By Theorem \ref{SV}, we have $h^0(\mathcal O_X(H+K))=\chi+\pi-1=6$. Moreover, the adjunction map
	$$\Phi = \Phi_{|H+K|}: X \to X_1\subseteq\Bbb P^{5}$$
is birational onto a smooth surface $X_1$ with intersection matrix 
$$\begin{pmatrix}
H_1^2 & H_1\cdot K_1\\
H_1\cdot K_1 & K_1^2
\end{pmatrix} =
\begin{pmatrix}
5 & -5\\
-5 & K_1^2
\end{pmatrix}. 
$$
By the Hodge index theorem and $(H_1+K_1)^2\ge 0$, we have $K_1^2= 5$. 
 It follows from the classification  of surfaces of degree $5$ (see \cite{Ion82} or \cite{Liv90}) that 
$X_1$ is a smooth Del Pezzo surface of degree $5$ with intersection matrix 
$$\begin{pmatrix}
H_1^2 & H_1\cdot K_1\\
H_1\cdot K_1 & K_1^2
\end{pmatrix} =
\begin{pmatrix}
5 & -5\\
-5 & 5
\end{pmatrix}. 
$$
Since $K_1^2-K^2=10$, $\Phi : X \to X_1$ is a simultaneous blow down of  ten $(-1)$--lines on $X$.
Since $K^2=-5$, the number of exceptional divisors of surface $Y$ is $9-K^2=14$.
It follows that  $X=\Bbb P^2(p_1,\ldots, p_{14})$ is a blowup of $\Bbb P^2$ in general fourteen  points  by the complete  linear system
$$H=(6;2^4,1^{10})=6L-\sum\limits_{i=1}^{4}2E_i-\sum\limits_{i=5}^{14}E_i.$$ Moreover, $X$ is smooth by Proposition \ref{DS} and $ \chi_{top} = 3+14=17$.

\subsection{\bf Surface with $K^2=-4$} 
The intersection matrix of $X$ is 
$$\begin{pmatrix}
H^2 & H\cdot K\\
H\cdot K & K^2
\end{pmatrix} =
\begin{pmatrix}
10 & 0\\
0 & -4
\end{pmatrix} .
$$
By Theorem \ref{SV}, we have $h^0(\mathcal O_X(H+K))=\chi+\pi-1=6$. Moreover, the adjunction map
	$$\Phi = \Phi_{|H+K|}: X \to X_1\subseteq\Bbb P^{5}$$
is birational onto a smooth surface $X_1$ with intersection matrix 
$$\begin{pmatrix}
H_1^2 & H_1\cdot K_1\\
H_1\cdot K_1 & K_1^2
\end{pmatrix} =
\begin{pmatrix}
6 & -4\\
-4 & K_1^2
\end{pmatrix}. 
$$
By the Hodge index theorem and $(H_1+K_1)^2\ge 0$, we have $K_1^2=2$. 
Therefore $X_1$ is a surface of degree $6$ and sectional genus $2$
with  intersection matrix 
$$\begin{pmatrix}
H_1^2 & H_1\cdot K_1\\
H_1\cdot K_1 & K_1^2
\end{pmatrix} =
\begin{pmatrix}
6 & -4\\
-4 & 2
\end{pmatrix}. 
$$
Thus $\Phi : X \to X_1$ is a simultaneous blow down of  six $(-1)$--lines on $X$.
It follows from the classification  of surfaces of degree $6$ (see \cite{Ion82} or \cite{Liv90}) that  $X_1\subseteq \Bbb P^5$ is a conic bundle $X_1\to \Bbb P^1$ with $6$ singular fibers. 
So we have the following two choices for $X$. 

\begin{enumerate}[{(i)}] \rm

\item First, $X_1=\Bbb P^2(p_1,\ldots, p_{7})$  is  a blowup 
of $\Bbb P^2$ in general seven points by the complete  linear system
$H_1=(4;2^1,1^{6})$, and so $X=\Bbb P^2(p_1,\ldots, p_{13})$
 is  a blowup of $\Bbb P^2$ in general thirteen  points by the complete  linear system
$$H=(7;3^1,2^{6},1^{6})=7L-3E_1-\sum\limits_{i=2}^72E_i-\sum\limits_{i=8}^{13}E_i.$$ 
Moreover, $X$ is smooth by Proposition \ref{DS} and $ \chi_{top} = 3+13=16$.

\item Second, $X_1=\Bbb F_e(p_2,\ldots, p_{7})$ is a blowup of Hirzebruch surface $\Bbb F_e$ in general six points, and so $X=\Bbb F_e(p_2,\ldots, p_{13})$ is  a blowup 
of  Hirzebruch surface $\Bbb F_e$ in general twelve points   by the complete  linear system
$$H_e=(4,2e+5;2^{6},1^{6})=4C_0+(2e+5)f-\sum\limits_{i=2}^72E_i-\sum\limits_{i=8}^{13}E_i.$$
Moreover, $0\le e\le 2$, since $H\cdot C_0\ge 1$. In the case $e=2$, $H$  is not a very ample linear system of dimension $5$. Therefore, $0\le e\le 1$. Note that for $e=0$, $1$, another description of the surface $X$ can be given, using a plane model, as follows.
\begin{enumerate}[{a)}] \rm
\item  ${\bf e = 0:}$ $\Bbb F_0(p_2,\ldots, p_{13})$ is isomorphic to the blowup of $\Bbb P^2$ at    general thirteen points by the complete  linear system $H=(7;3^1,2^{6},1^{6})$. In fact, $\Bbb F_0$ is isomorphic to the quadric surface $Q \subseteq\Bbb P^3$ and $Q$ is obtained by the complete (not very ample) linear system $|2L - E_1 - E_2|$. Since the line $|L-E_1-E_2|$  is contracted to a point $P\in Q$, we have that to blowup $Q$ at the point $P$ and at other general eleven points is equivalent to a blowup of $\Bbb P^2$ at    general thirteen points, as we said. Taking $C_0 = L - E_2$ and $f = L - E_1$, to the very ample divisor $H_0 = 4C_0+5f-\sum\limits_{i=2}^72E_i-\sum\limits_{i=8}^{13}E_i$ on $\Bbb F_0(p_2,\ldots, p_{13})$ corresponds the very ample divisor 
$$\begin{aligned}
H_0 &= 4(L -E_2) +5(L -E_1)-2(L -E_1 -E_2)-\sum\limits_{i=3}^72E_i-\sum\limits_{i=8}^{13}E_i \\&
= 7L-3E_1-\sum\limits_{i=2}^72E_i-\sum\limits_{i=8}^{13}E_i
\end{aligned}$$
on $X=\Bbb P^2(p_1,\ldots, p_{13})$, which is the same divisor we found in (i).

\item ${\bf e=1:}$ 
Let us determine the very ample divisor $H =tL-mE_1 - \sum\limits_{i=2}^72E_i-\sum\limits_{i=8}^{13}E_i$ on $X$ which corresponds to the divisor $4C_0+7f-\sum\limits_{i=2}^72E_i-\sum\limits_{i=8}^{13}E_i$ on $\Bbb F_1(p_2,\ldots, p_{13})$. The
integers $t , m > 0$ are such that
$$\begin{cases} \binom{t+2}{2}-\binom{m+1}{2}-3\cdot 6-6=6 \\
t^2-m^2-30=10.
 \end{cases}
$$
Solving the equations we find $t = 7$ and $m = 3$. Hence $H =7L-3E_1-\sum\limits_{i=2}^72E_i-\sum\limits_{i=8}^{13}E_i$, which is the same divisor we found in (i).
\end{enumerate}

\end{enumerate}

\subsection{\bf Surface with $K^2=-3$} 
The intersection matrix of $X$ is 
$$\begin{pmatrix}
H^2 & H\cdot K\\
H\cdot K & K^2
\end{pmatrix} =
\begin{pmatrix}
10 & 0\\
0 & -3
\end{pmatrix}. 
$$
By Theorem \ref{SV}, we have $h^0(\mathcal O_X(H+K))=\chi+\pi-1=6$. Moreover, the adjunction map
	$$\Phi = \Phi_{|H+K|}: X \to X_1\subseteq\Bbb P^{5}$$
is birational onto a smooth surface $X_1$ with intersection matrix 
$$\begin{pmatrix}
H_1^2 & H_1\cdot K_1\\
H_1\cdot K_1 & K_1^2
\end{pmatrix} =
\begin{pmatrix}
7 & -3\\
-3 & K_1^2
\end{pmatrix}. 
$$
By the Hodge index theorem and $(H_1+K_1)^2\ge 0$, we have $-1\le K_1^2\le 1$. 
It follows from the classification  of surfaces of degree $7$ (see \cite{Ion82} or \cite{Liv90}) that we have  three cases.

\begin{enumerate}[{(i)}] \rm
\item If $K^2_1=1$ then $X_1$ is a smooth surface of degree $7$, sectional genus $3$ with intersection matrix 
$$\begin{pmatrix}
H_1^2 & H_1\cdot K_1\\
H_1\cdot K_1 & K_1^2
\end{pmatrix} =
\begin{pmatrix}
7 & -3\\
-3 & 1
\end{pmatrix}. 
$$
Thus $\Phi : X \to X_1$ is a simultaneous blow down of  four $(-1)$--lines on $X$.
Moreover, by Theorem \ref{SV}, 
$X_1=\Bbb P^2(p_1,\ldots, p_{8})$
 is  a blowup  of $\Bbb P^2$ in  general eight  points by the complete  linear system
$H_1=(6;2^7,1^{1})$. Hence, $X=\Bbb P^2(p_1,\ldots, p_{12})$ is  a blowup 
of $\Bbb P^2$ in  twelve  points in general position by the complete  linear system
$$H=(9;3^7,2^{1},1^{4})=9L-\sum\limits_{i=1}^{7}3E_i-2E_8-\sum\limits_{i=9}^{12}E_i.$$
Since $X$ has maximum degree $10=\binom{5}{2}$, to show that the surface $X$ does not satisfy property ${\bf N}_{3,3}$  
it is sufficient to verify that $h^0(\mathcal I_X(2))=1$.  We verified this via Macaulay2 and we can conclude that $h^0(\mathcal I_X(2))=1$ (see \cite{Tr2019}).

\item If $K^2_1=0$ then $X_1$ is a smooth surface of degree $7$, sectional genus $3$ with intersection matrix 
$$\begin{pmatrix}
H_1^2 & H_1\cdot K_1\\
H_1\cdot K_1 & K_1^2
\end{pmatrix} =
\begin{pmatrix}
7 & -3\\
-3 & 0
\end{pmatrix}. 
$$
Thus $\Phi : X \to X_1$ is a simultaneous blow down of  three $(-1)$--lines on $X$.
By Theorem \ref{SV}, we have $h^0(\mathcal O_{X_1}(H_1+K_1))=\chi(X_1)+\pi_{X_1}-1=3$. Moreover, the adjunction map
	$$\Phi^1 = \Phi^1_{|H_1+K_1|}: X_1 \to \Bbb P^2$$
is birational onto a plane $\Bbb P^2$. 
Since $K^2=-3$, the number of exceptional divisors of surface $X$ is $9-K^2=12$. Thus 
$\Phi^1 : X_1 \to \Bbb P^2$ is a simultaneous blow down of  $12-3=9$ many $(-1)$--lines on $X_1$.
Therefore $X=\Bbb P^2(p_1,\ldots, p_{12})$
 is  a blowup 
of $\Bbb P^2$ in general twelve  points by the complete  linear system
$$H=(7;2^{9},1^{3})=7L-\sum\limits_{i=1}^{9}2E_i-\sum\limits_{i=10}^{12}E_i.$$

\item If $K^2_1=-1$ then $X_1$ is a smooth surface of degree $7$, sectional genus $3$ with intersection matrix 
$$\begin{pmatrix}
H_1^2 & H_1\cdot K_1\\
H_1\cdot K_1 & K_1^2
\end{pmatrix} =
\begin{pmatrix}
7 & -3\\
-3 & -1
\end{pmatrix}. 
$$
Thus $\Phi : X \to X_1$ is a simultaneous blow down of  two $(-1)$--lines on $X$.
Therefore $X=\Bbb F_e(p_1,\ldots, p_{11})$ is  a blowup 
of  Hirzebruch surface $\Bbb F_e$ in general eleven  points by the complete  linear system
$$H=(4,2e+6;2^{9},1^{2})=4C_0+(2e+6)f-\sum\limits_{i=1}^92E_i-E_{10}-E_{11}.$$
Morveover $0\le e\le 2$ since $H\cdot C_0\ge 1$. 
Since $X$ has maximum degree $10=\binom{5}{2}$, to show that the surface $X$ does not satisfy property ${\bf N}_{3,3}$  
it is sufficient to verify that $h^0(\mathcal I_X(2))=1$.  We verified this via Macaulay2 and we can conclude that $h^0(\mathcal I_X(2))=1$ (see \cite{Tr2019}).

\end{enumerate}

\subsection{\bf Surface with $K^2=-2$} 
The intersection matrix of $X$ is 
$$\begin{pmatrix}
H^2 & H\cdot K\\
H\cdot K & K^2
\end{pmatrix} =
\begin{pmatrix}
10 & 0\\
0 & -2
\end{pmatrix}. 
$$
By Theorem \ref{SV}, the adjunction map
	$$\Phi = \Phi_{|H+K|}: X \to X_1\subseteq\Bbb P^{5}$$
is birational onto a smooth surface $X_1$ with intersection matrix 
$$\begin{pmatrix}
H_1^2 & H_1\cdot K_1\\
H_1\cdot K_1 & K_1^2
\end{pmatrix} =
\begin{pmatrix}
8 & -2\\
-2 & K_1^2
\end{pmatrix}. 
$$
By the Hodge index theorem and $(H_1+K_1)^2\ge 0$, we have $-2\le K_1^2\le 0$. 
Therefore we have three cases.

\begin{enumerate}[{(i)}] \rm
\item If $K^2_1=0$ then $X_1$ is a smooth surface of degree $8$, sectional genus $4$ with intersection matrix 
$$\begin{pmatrix}
H_1^2 & H_1\cdot K_1\\
H_1\cdot K_1 & K_1^2
\end{pmatrix} =
\begin{pmatrix}
8 & -2\\
-2 & 0
\end{pmatrix}. 
$$
Thus $\Phi : X \to X_1$ is a simultaneous blow down of  two $(-1)$--lines on $X$.
By Theorem \ref{SV}, we have $h^0(\mathcal O_{X_1}(H_1+K_1))=\chi(X_1)+\pi_{X_1}-1=4$. Moreover,  the adjunction map
	$$\Phi^1 = \Phi^1_{|H_1+K_1|}: X_1 \to X_2\subseteq\Bbb P^{3}$$
is birational onto a smooth surface $X_2$ of degree $4$ with intersection matrix 
$$\begin{pmatrix}
H_2^2 & H_2\cdot K_2\\
H_2\cdot K_2 & K_2^2
\end{pmatrix} =
\begin{pmatrix}
4 & -2\\
-2 & K_2^2
\end{pmatrix}. 
$$
By the Hodge index theorem and $(H_2+K_2)^2\ge 0$, we have $K_2^2\in\{0,1\}$. By Theorem \ref{SV} and Proposition \ref{P2.5}, there doesn't exist surface $X_2$.

\item If $K^2_1=-1$ then $X_1$ is a smooth surface of degree $8$, sectional genus $4$ with intersection matrix 
$$\begin{pmatrix}
H_1^2 & H_1\cdot K_1\\
H_1\cdot K_1 & K_1^2
\end{pmatrix} =
\begin{pmatrix}
8 & -2\\
-2 & -1
\end{pmatrix}. 
$$
Thus $\Phi : X \to X_1$ is a simultaneous blow down of a $(-1)$--line on $X$.
By Theorem \ref{SV}, we have $h^0(\mathcal O_{X_1}(H_1+K_1))=\chi(X_1)+\pi_{X_1}-1=4$. Moreover, the adjunction map
	$$\Phi^1 = \Phi^1_{|H_1+K_1|}: X_1 \to X_2\subseteq\Bbb P^{3}$$
is birational onto a smooth cubic surface $X_2$ with intersection matrix 
$$\begin{pmatrix}
H_2^2 & H_2\cdot K_2\\
H_2\cdot K_2 & K_2^2
\end{pmatrix} =
\begin{pmatrix}
3 & -3\\
-3 & 3
\end{pmatrix}. 
$$
By the Hodge index theorem and $(H_2+K_2)^2\ge 0$, we have $K_2^2=3$.
Thus $\Phi^1 : X_1 \to X_2$ is a simultaneous blow down of  four $(-1)$--lines on $X_1$.
Since $K_{1}^2=-1$, the number of exceptional divisors of surface $X_1$ is $9-K_1^2=10$.
Therefore  $X=\Bbb P^2(p_1,\ldots, p_{11})$ is  a blowup 
of $\Bbb P^2$ in general eleven  points by the complete  linear system
$$H=(9;3^{6},2^{4},1^{1})=9L-\sum\limits_{i=1}^{6}3E_i-\sum\limits_{i=7}^{10}2E_i-E_{11}.$$

\item If $K^2_1=-2$ then $X_1$ is a smooth surface of degree $8$, sectional genus $4$ with  intersection matrix 
$$\begin{pmatrix}
H_1^2 & H_1\cdot K_1\\
H_1\cdot K_1 & K_1^2
\end{pmatrix} =
\begin{pmatrix}
8 & -2\\
-2 & -2
\end{pmatrix}. 
$$
By Theorem \ref{SV}, we have $h^0(\mathcal O_{X_1}(H_1+K_1))=\chi(X_1)+\pi_{X_1}-1=4$. Moreover, the adjunction map
	$$\Phi^1 = \Phi^1_{|H_1+K_1|}: X_1 \to X_2\subseteq\Bbb P^{3}$$
is birational onto a smooth quadric surface $X_2$ with intersection matrix 
$$\begin{pmatrix}
H_2^2 & H_2\cdot K_2\\
H_2\cdot K_2 & K_2^2
\end{pmatrix} =
\begin{pmatrix}
2 & -4\\
-4 & K_2^2
\end{pmatrix}. 
$$
By the Hodge index theorem and $(H_2+K_2)^2\ge 0$, we have $K_2^2\in\{6,7,8\}$. 
By Proposition \ref{P2.5}, we have the following two choices for $X$. 

\begin{enumerate}[{(a)}] \rm

\item  First  $X=\Bbb P^2(p_1,\ldots, p_{11})$ is  a blowup 
of $\Bbb P^2$ in general eleven  points  by the complete  linear system
$$H=(8;3^{2},2^{9},1^{0})=8L-3E_0-3E_1-\sum\limits_{i=2}^{11}2E_i.$$
 Since $X$ has maximum degree $10=\binom{5}{2}$, to show that the surface $X$ does not satisfy property ${\bf N}_{3,3}$   
 it is sufficient to verify that $h^0(\mathcal I_X(2))=1$.  We verified this via Macaulay2 and we can conclude that $h^0(\mathcal I_X(2))=1$ (see \cite{Tr2019}).

\item Second  $X=\Bbb F_e(p_2,\ldots, p_{11})$ is  a blowup 
of  Hirzebruch surface in general ten points by the complete  linear system
$$H=(5,2e+5;2^{10})=5C_0+(2e+5)f-\sum\limits_{i=2}^{11}2E_i.$$
Morveover $e=0$ since $H\cdot C_0\ge 1$. 
Moreover, $\Bbb F_0(p_2,\ldots, p_{11})$ is isomorphic to the blowup of $\Bbb P^2$ at general eleven  points by the complete  linear system $H=(8;3^2,2^{9},1^{0})$. In fact, $\Bbb F_0$ is isomorphic to the quadric surface $Q \subseteq\Bbb P^3$ and $Q$ is obtained by the complete (not very ample) linear system $|2L - E_1 - E_2|$. Since the line $|L-E_1-E_2|$  is contracted to a point $P\in Q$, we have that to blowup $Q$ at the point $P$ and at other  general eleven  points is equivalent to a blowup of $\Bbb P^2$ at  general thirteen  points, as we said. Taking $C_0 = L - E_2$ and $f = L - E_1$, to the very ample divisor $H_0 = 5C_0+5f-\sum\limits_{i=2}^{10}2E_i$ on $\Bbb F_0(p_2,\ldots, p_{11})$ corresponds the very ample divisor 
$$\begin{aligned}
H_0 &= 5(L -E_2) +5(L -E_1)-2(L -E_1 -E_2)-\sum\limits_{i=3}^{11}2E_i \\&
= 8L-3E_1-3E_2-\sum\limits_{i=3}^{11}2E_i
\end{aligned}$$
on $X=\Bbb P^2(p_1,\ldots, p_{11})$, which is the same divisor we found in (a). Thus surface $X$ is not Cohen-Macaulay.
\end{enumerate}

\end{enumerate}

\subsection{\bf Surface with $K^2=-1$} 
The intersection matrix of $X$ is 
$$\begin{pmatrix}
H^2 & H\cdot K\\
H\cdot K & K^2
\end{pmatrix} =
\begin{pmatrix}
10 & 0\\
0 & -1
\end{pmatrix}. 
$$
By Theorem \ref{SV}, the adjunction map
	$$\Phi = \Phi_{|H+K|}: X \to X_1\subseteq\Bbb P^{5}$$
is birational onto a smooth surface $X_1$ with intersection matrix 
$$\begin{pmatrix}
H_1^2 & H_1\cdot K_1\\
H_1\cdot K_1 & K_1^2
\end{pmatrix} =
\begin{pmatrix}
9 & -1\\
-1 & K_1^2
\end{pmatrix}. 
$$
By the Hodge index theorem and $(H_1+K_1)^2\ge 0$, we have $K_2^2\in\{-1,0\}$.
 Therefore we have two cases.

\begin{enumerate}[{(i)}] \rm
\item If $K^2_1=0$ then $X_1$ is a smooth surface of degree $9$, sectional genus $5$ with  intersection matrix 
$$\begin{pmatrix}
H_1^2 & H_1\cdot K_1\\
H_1\cdot K_1 & K_1^2
\end{pmatrix} =
\begin{pmatrix}
9 & -1\\
-1 & 0
\end{pmatrix}. 
$$
By Theorem \ref{SV}, we have $h^0(\mathcal O_{X_1}(H_1+K_1))=\chi(X_1)+\pi_{X_1}-1=5$. Moreover, the adjunction map
	$$\Phi = \Phi^1_{|H_1+K_1|}: X_1 \to X_2\subseteq\Bbb P^{4}$$
is birational onto a smooth surface $X_2$ with intersection matrix 
$$\begin{pmatrix}
H_2^2 & H_2\cdot K_2\\
H_2\cdot K_2 & K_2^2
\end{pmatrix} =
\begin{pmatrix}
7 & -1\\
-1 & K^2_2
\end{pmatrix}. 
$$
By the Hodge index theorem and $K_2^2\ge K_1^2$, we have $K_2^2 =0$.
Since $X_2$ is a smooth surface of degree $7$ in $\Bbb P^4$ with normal bundle $\mathcal N_{X_2}$ 
the numerical invariants of $X_2$ satisfy the double point formula
$$d_{X_2}^2-10d_{X_2}-5H_2\cdot K_2-2K_2^2+12\chi_{X_2}=0$$
(see \cite[Appendix A, 4.1.3]{Har77}). Thus we have
$3\chi_{X_2}=4$, impossible.

\item If $K^2_1=-1$ then $X_1$ is a smooth surface of degree $9$, sectional genus $5$ with  intersection matrix 
$$\begin{pmatrix}
H_1^2 & H_1\cdot K_1\\
H_1\cdot K_1 & K_1^2
\end{pmatrix} =
\begin{pmatrix}
9 & -1\\
-1 & -1
\end{pmatrix}. 
$$
By Theorem \ref{SV}, we have $h^0(\mathcal O_{X_1}(H_1+K_1))=\chi(X_1)+\pi_{X_1}-1=5$. Moreover, the adjunction map
	$$\Phi^1= \Phi^1_{|H_1+K_1|}: X_1 \to X_2\subseteq\Bbb P^{4}$$
is birational onto a smooth cubic surface $X_2$ with intersection matrix 
$$\begin{pmatrix}
H_2^2 & H_2\cdot K_2\\
H_2\cdot K_2 & K_2^2
\end{pmatrix} =
\begin{pmatrix}
6 & -2\\
-2 & -1
\end{pmatrix}. 
$$
Since $X_2$ is a smooth surface of degree $6$ in $\Bbb P^4$ with normal bundle $\mathcal N_{X_2}$ there is the relation
$$d_{X_2}^2-10d_{X_2}-5H_2\cdot K_2-2K_2^2+12\chi_{X_2}=0.$$
Thus we have
$6\chi_{Y_1}=7+K_2^2$. Thus $K_2^2=-1$.
By Theorem \ref{SV}, we have $h^0(\mathcal O_{X_2}(H_2+K_2))=\chi(X_2)+\pi_{X_2}-1=3$. Moreover,  the adjunction map
	$$\Phi^2 = \Phi^2_{|H_2+K_2|}: X_2 \to \Bbb P^{2}$$
is birational onto a plane $\Bbb P^2$. 
Since $K^2=-1$, the number of exceptional divisors of surface $X$ is $9-K^2=10$. Thus 
$\Phi^2 : X _2\to \Bbb P^2$ is a simultaneous blow down of  ten $(-1)$--lines on $X_2$.
Therefore  $X=\Bbb P^2(p_1,\ldots, p_{10})$ is  a blowup 
of $\Bbb P^2$ in general ten  points  by the complete  linear system
$$H=(10;3^{10})=10L-\sum\limits_{i=1}^{10}3E_i.$$

\end{enumerate}

\subsection{\bf Surface with $K^2=0$} 
Surface $X$ is a smooth Enriques surface of degree $10$ in $\Bbb P^5$. 
 Smooth Enriques are well understood, in fact any Enriques surface has a linear system of degree $10$ and projective dimension $5$ without base points, for very ampleness it suffices to require that any elliptic curve on the surface has degree at least $3$ with respect to the linear system, and that there are no rational curves on it. These surfaces have exactly $20$ plane cubic curves whose planes form the union of the
trisecants to the surface. Thus projecting the
surface from a point on the surface away from these $20$ planes, we get a smooth surface $X^\prime$ of
degree $9$, sectional genus $6$ in $\Bbb P^4$ (see \cite{DES93}).  The goal of this subsection is to construct a surface $X^\prime$ of degree $9$, sectional genus $6$ in $\Bbb P^4$. To do so, we will use work of Rao, who showed that the construction of curves is equivalent to the creation of its Hartshorne--Rao module.
\begin{defn} The Hartshorne--Rao module of $X^\prime $ is the finite length module
$$M=M_{X^\prime}:=\bigoplus\limits_{n\in\Bbb Z}H^1(\Bbb P^4,\mathcal{I}_{X^{\prime}}(n))\subseteq \bigoplus\limits_{n\in\Bbb Z}H^0(\Bbb P^4,\mathcal{O}(n))\cong R^\prime:=k[x_0,\ldots,x_4].$$
\end{defn}

	Assuming the open condition that $X^\prime$ has maximal rank, i.e.
	 restriction map $H^0(\mathcal{O}_{\Bbb P^4}(m))\to H^0(\mathcal{O}_{X^\prime}(m))$ has maximal rank for all $m$.
Then Hilbert series of the Hartshorne-Rao module is $H_M(t) = 3t^2 + 10t^3 + 6t^4$ 
	and the Hilbert numerator is
	$$(1-t)^4H_M(t)=3t^2-10t^3 +6t^4 +15t^5-25 t^6+12t^7-t^9  .$$
	If $M$ has a natural resolution, which means that for
each degree $j$ at most one $\beta_{ij}$ is nonzero then $M$ has a Betti table
\begin{center}
		\begin{tabular}{ c | c c c c c c c c c}
			&   & $0$ & $1 $ &$2 $  &$3 $&$4$ &$5$\\ 
			\hline
			2&  & $3    $ & $10$ & $6$ & $\cdot$&$\cdot$ &$\cdot$ \\ 
			3&  & $\cdot$ & $\cdot$ &$15$ &$25$&$12$ &$\cdot$ \\ 
			4&  & $\cdot$ & $\cdot$ &$\cdot$ &$\cdot$&$\cdot$&$1$  \\ 
			
		\end{tabular}.
\end{center}
Note that having a natural resolution is an open condition in a family of modules with constant Hilbert function. If homogeneous coordinate ring $R^\prime_{X^\prime}=R^\prime/I_{X^\prime}$ and the section ring $\Gamma_*(\mathcal{O}_{X^\prime}):=\bigoplus\limits_{n\in\Bbb Z}H^i(\mathcal{O}_{X^\prime}(n))$ has natural resolution as well, then their Betti tables will be
	\begin{center}
		\begin{tabular}{ c | c c c c c c c c c}
			&   & $0$ & $1 $ &$2 $  &$3 $&$4$ \\ 
			\hline
			0&  & $1    $ & $\cdot$ & $\cdot$ & $\cdot$&$\cdot$  \\ 
			1&  & $\cdot$ & $\cdot$ &$\cdot$ &$\cdot$&$\cdot$  \\ 
			2&  & $\cdot$ & $\cdot$ &$\cdot$ &$\cdot$&$\cdot$  \\ 
			3&  & $\cdot$ & $\cdot$ &$\cdot$ &$\cdot$&$\cdot$  \\ 
			4&  & $\cdot$ & $15    $ &$25$ &$12$&$\cdot$  \\ 
			5&  & $\cdot$ & $ \cdot$ &$  \cdot $ &$ \cdot$&$  1  $ \\ 
		\end{tabular}
	\qquad\qquad and\qquad\qquad
		\begin{tabular}{ c | c c c c c c c c c}
	&   & $0$ & $1 $ &$2 $  \\
	\hline
	0&  & $1    $ & $\cdot$ & $\cdot$ & \\ 
	1&  & $\cdot$ &$\cdot$ &$\cdot$\\ 
	2&  & $  3  $ &$  10 $ &$ 6  $ \\ 
\end{tabular}.
		
	\end{center}

To construct a desired $M$, we start with the matrix $\psi$  defining 	$\xymatrix{{R^\prime}^{12}(-3) &\ar[l]_{\psi} {R^\prime}(-5)}$
which we choose randomly. Then we get a smooth surface $X^\prime$ of degree $9$, sectional genus $6$ in $\Bbb P^4$. Moreover the intersection matrix of $X^\prime$ is 
$$\begin{pmatrix}
H_{X^\prime}^2 & H_{X^\prime}\cdot K_{X^\prime}\\
H_{X^\prime}\cdot K_{X^\prime} & K_{X^\prime}^2
\end{pmatrix} =
\begin{pmatrix}
9 & 1\\
1 & -1
\end{pmatrix}.
$$
 Notice that $X^\prime$ is  regular and $|2K_{X^\prime}|\neq 0$, then $2K_{X^\prime} =2 E_{X^\prime}$. Thus $X^\prime$ is an Enriques surface blown up in one point.  By Theorem \ref{SV}, the adjunction map
	$$\Phi^\prime = \Phi^\prime_{|H_{X^\prime}+K_{X^\prime}|}: X^\prime \to X\subseteq\Bbb P^{5}$$
is birational onto a smooth surface $X$ of degree $10$ with intersection matrix 
$$\begin{pmatrix}
H^2 & H\cdot K\\
H\cdot K & K^2
\end{pmatrix} =
\begin{pmatrix}
10 & 0\\
0 & 0
\end{pmatrix}. 
$$

We summarize the above discussion in the following proposition.

\begin{prop}\label{clas}
Let $X$ be a non-degenerate projective surface in $\Bbb P^{5}$. Suppose that  $X$ satisfies property ${\bf N}_{3,3}$ and $\deg(X)=10$. Then $X$ is a projective smooth Cohen-Macaulay surface with the minimal free resolution as  in \ref{eqn},
%
and $X$ is an image of $\Bbb P^2$ under a rational map, defined by a complete linear series as in Table \ref{table:S106} or a Fano model of  an Enriques surface. Moreover, $X$ defines a $9$-dimensional family of cubic fourfolds $Y$ with $X\subset Y$. 
\end{prop}
Note in the proof of above Proposition, we claimed that there exists an example of  surfaces $X$ of degree $10$ and sectional genus $6$ which is also embedded in $\Bbb P^5$. In an ancillary file (see \cite{Tr2019}), we provide the explicit homogeneous ideals of such surfaces.  Notice that last statement of above Proposition is an immediate corollary of $h^0(\mathcal I_X(3))=10$.

To prove Theorem \ref{theorem1}, we need to prove the following result.
\begin{prop}\label{P3.2}
	Let $X$ be a general projective Cohen-Macaulay surface of degree $10$, sectional genus $6$.		 Then $X$ is of  maximal rank and	
	homogeneous coordinate ring $R_Y=R/I_X$ and the section ring $\Gamma_*(\mathcal{O}_X)$ have minimal free resolutions with the Betti tables as  in \ref{eqn}.

%
%

\end{prop}
\begin{proof}
	Assuming that restriction map $H^0(\mathcal{O}_{\Bbb P^5}(m))\to H^0(\mathcal{O}_{X}(m))$ has maximal rank. Then we have the following statements.
	\begin{itemize}

\item	The Hilbert series of the homogeneous coordinate ring of $X$ is
	$$H_X(t) = 1 + 6t + 21t^2 + 46 t^3  + 81 t^4  + 126 t^5  +181 t^6 +246 t^7 +\ldots.$$ 
\item	
	The Hartshorne--Rao module 
	$$H^1_*(\mathcal{I}_X)=H^2_*(\mathcal{I}_X)=0,$$
	where $H^i_*(\mathcal{I}_X)=\bigoplus\limits_{n\in\Bbb Z}H^i(\mathcal{I}_X(n))$ for all $i\in \Bbb Z$.
\item	
	The Hilbert numerator has shape
	$$(1-t)^5H_X(t)=1-10t^3 +15t^4-6 t^5.$$
	\end{itemize}
	
	Since $H_X(t)=H_{\Gamma_*(\mathcal{O}_X)}(t)$, the Betti table  of the resolution $F^\bullet$ of the section ring $\Gamma_*(\mathcal{O}_X)$ as an $R$-module is as in \ref{eqn}.
To show that the Betti tables are indeed the expected ones and that a general surface $Y$ is of maximal rank, we only need to exhibit a concrete example, which we construct  in Proposition \ref{clas}.	
\end{proof}
Let us give a proof of Theorem \ref{theorem1}.
\begin{tpf1}
$1)\Rightarrow 2)$  follows from Theorem 1.2 in \cite{AhK15}.\\
$2)\Rightarrow 3)$ is trivial \\
$3)\Rightarrow 1)$ follows from Proposition \ref{P3.2}.\\
The next result follows immediately from the Proposition  \ref{clas}.
\end{tpf1}
\begin{cor}
The moduli space of  projective smooth Cohen-Macaulay surfaces of degree $10$, sectional genus $6$ have exactly $7$ irreducible components.
\end{cor}

\section{Connection with cubic fourfolds}\label{Cdelta}
The goal of this section is to give the relation between  projective surfaces of degree $10$ satisfying  property  ${\bf N}_{3,3}$ and special cubic fourfolds.
Recall that a smooth cubic fourfold  $Y$, the vanishing locus of a degree $3$ homogeneous polynomial in $6$ variables, is called {\it special}, if the existence of an embedding of a saturated rank-$2$ lattice 
$$L_\delta:=\langle h^2,X\rangle\hookrightarrow A(Y) ,$$
where $A(Y)$ is the lattice of middle Hodge classes, $h\in \Pic(Y)$  is the hyperplane class,  $X$ is an algebraic surface not homologous to a complete intersection, and $\delta$ is the determinant of the intersection matrix of $\langle h^2,X\rangle$. In 2000, Hassett showed in \cite{Has00} that the locus  $\mathcal C_\delta$ of special cubic fourfolds  $Y$ of a discriminant $\delta$ is an irreducible divisor which is nonempty if and only if $\delta>6$ and $\delta\equiv 0,2 \pmod 6 $. Recently, the study of divisors $\mathcal C_\delta$  for small $\delta$   has received considerable attention. In particular, we systematically study concrete descriptions of special cubic fourfolds of small discriminants. For $\delta=8$, $12$, $14$, $20$, Hassett showed that a generic member $Y\in \mathcal C_\delta$ could be described explicitly. The surfaces  $X$ in these cases were given by planes, cubic scrolls, quintic del Pezzos, and Veronese surfaces, respectively. For some choices of $\delta$, the generic $Y\in \mathcal C_\delta$  admits an alternative description. For example, the generic $Y\in \mathcal C_8$ can be described as containing an octic K3 surface, and the generic $Y\in \mathcal C_{14}$
containing  a quartic scroll. In \cite{Nue17}, Nuer proved that the generic element $Y$ of $\mathcal C_\delta$ for $12\le \delta\le 38$ contains a smooth rational surface obtained as the blowup of $\Bbb P^2$ at generic points $p$, and the generic $Y\in \mathcal C_{44}$   contains  a Fano model of  an Enriques surface. 
Method of Nuer certainly indicates that the method presented here has been exhausted for the most part.
On the other hand, it is worth noting that Voisin has shown that for a special cubic fourfold $Y$, the lattice $A(Y)$ can be generated either by smooth surfaces \cite[proof of Theorem 5.6]{Voi17} or by possibly singular rational surfaces \cite{Voi07}, and many have wondered if $A(Y)$ is in fact generated by {\it smooth} rational surfaces. This has certainly been true for the few low-discriminant cases previously known, and the work of Nuer provides a lot more evidence for this possibility. Therefore, Nuer gave the following question.

\begin{ques}
What does the condition that the generic $Y \in \mathcal C_\delta$ contains  a smooth rational surface not homologous to a complete intersection say, if anything, about the geometry of $Y$? and about the geometry of $\mathcal C_\delta$?
\end{ques}


The following theorem  gives a partial answer to this question.

\begin{thm}
Let $Y$ be a smooth cubic fourfold. Then the statements are equivalent. 
\begin{itemize}
\item[$1)$]  $Y$ contains   a non-degenerate  projective surface $X$ of degree $10$ satisfies property  ${\bf N}_{3,3}$.
\item[$2)$]  $Y\in \mathcal{C}_\delta$ for $6<\delta\le 44$ and $\delta\equiv 2 \pmod 6$.
\end{itemize}
\end{thm} 
\begin{proof}
$1)\Rightarrow 2) $ Assume that the cubic fourfold $Y$ contains a  projective surface $X \subset \Bbb P^5$  of degree $10$ satisfies property  ${\bf N}_{3,3}$. Let  $t=K_X^2$. By Lemma \ref{lem1}, we have $-6\le t\le 0$. Then by computing the self-intersection of $X \subseteq Y$ using the formula from \cite[]{Has00} and Proposition \ref{clas}, we have
$$X^2 =c_2(\mathcal N_{X/Y})=6H^2 +3H\cdot K +K^2 -\chi_{top}=48+2t.$$
Therefore, $\delta=3X^2-10^2=44+6t$. 

$2)\Rightarrow 1) $ 
Assume that $X$ is a  projective surface of degree $10$ satisfies property  ${\bf N}_{3,3}$. Set $t=K_X^2$.
Let  $\mathcal H_t$ be the irreducible component of the Hilbert scheme $\Hilb_{\Bbb P^5}^{\chi(\mathcal O_X(t))} $ containing $[X]$. Then $\dim \mathcal H_t=2\cdot (9-t)+27=45-2t$, and $\mathcal H_t$ is generically smooth. Indeed, one can verify that $h^1(\mathcal{N}_{X/\Bbb P^5}) = 0$.  Let $U \subseteq |H^0(\Bbb P^5(3)) | \cong \Bbb P^{55}$ be the open set corresponding to smooth cubic fourfolds. Since $h^0(\mathcal I_X(3)) = 10$, the locus
$ \mathcal{D}_{t} = \{([X], [Y]) : X \subseteq Y\} \subseteq \mathcal{H}_t\times U$
has dimension $45-2t + 10 -1 = 54-2t$. The image of $\pi_2 : \mathcal{D}_t \to U$ has dimension at most $54$, because the general cubic fourfold does not contain any $X$ belonging to $ \mathcal H_t$. For every $[Y] \in  \pi_2(\mathcal{D}_t)$, we have
$$\dim(\pi^{-1}_2 ([Y]))
 \ge \dim(\mathcal{D}_{t} ) - \dim(\pi_2(\mathcal{D}_t )) = 54-2t - \dim(\pi_2(\mathcal{C}_{\delta})) \ge 54-2t - 54 = -2t,$$
 for $-6\le t\le 0$.
Since $h^0(\mathcal{N}_{X/Y}) \ge \dim_{[X\subset Y]}(\pi^{-1}
_2 ([Y]))$, for every $[X\subset Y] \in \pi^{-1}_2 ([Y])$, in order to show that a general
$Y \in \mathcal{C}_{\delta}$ contains a surface $X$ it is sufficient to verify that $h^0(\mathcal{N}_{X/Y}) = -2t$, for a general smooth cubic fourfolds $Y \in |H^0(\mathcal{I}_X(3))|$, see also \cite{Nue17} for a similar argument.
We verified this via Macaulay 2  (\cite{GSM2}), and we can conclude that $h^0(\mathcal{N}_{X/Y})=-2t$ for $-6\le t\le 0$.
\end{proof}



\section{ACM bundles}

Our goal in this section is to study  the relationship between ACM bundles and   projective surfaces of degree $10$ satisfying property ${\bf N}_{3,3}$.  A vector bundle $\mathcal F$ on a projective
variety $Y$ is ACM if all its intermediate cohomology groups $H^i(Y; \mathcal F(m))$ are zero for
$0 < i < \dim Y$ and all $m\in \Bbb Z$. To associate ACM bundles to surfaces we work with $\mathcal F$-type resolutions (or so-called Bourbaki sequences,  c.f. \cite{BHU87}). We recall the definition here.

\begin{defn} Let $Y \subset \Bbb P^n$ be an equidimensional scheme,  and $X \subset Y$ be a codimension $2$ subscheme without embedded components. An $\mathcal F$-type resolution of $X$ on $Y$ is an exact sequence
$$\xymatrix{0\ar[r]&\mathcal L\ar[r] &\mathcal F \ar[r] &\mathcal{I}_{X, Y} \ar[r] &0}, $$
with $\mathcal L$ is dissoci\'{e},  and $\mathcal{F}$ is a coherent sheaf satisfying $H_\bullet^1(\mathcal F^\vee):=  \bigoplus\limits_{n\in\Bbb Z}H^1(\mathcal F^\vee(n))= 0$,  and $\Ext^1(\mathcal F, \mathcal{O}_Y) = 0.$
\end{defn}
When $Y$ satisfies Serre’s condition $S_2$,  and $H_\bullet^1(\mathcal{O}_Y) = 0$,  then an $\mathcal F$-type resolution of $X$ exists (see \cite[2.12]{Har03}).  An $\mathcal F$-type resolution on $Y$ is not unique but it is well known that any two $\mathcal F$-type resolutions of the same subscheme are stably equivalent (see \cite[1.10]{Har03}). In other words,  if $\mathcal F$  and $\mathcal F^\prime$ are two sheaves appearing in the middle of an $\mathcal F$-type resolution of a subscheme $X$,  then there exist dissoci\'{e} sheaves $\mathcal L_1$,  $\mathcal L_2$,  and an integer $a$ such that
$$\mathcal F \oplus \mathcal L_1 \cong \mathcal F^\prime(a)\oplus \mathcal L_2.$$

The following theorem and its converse involve the relationship between ACM bundles and   projective surfaces of degree $10$ satisfying  property ${\bf N}_{3,3}$.
\begin{thm}[{cf. \cite[Theorem 3.2]{TrY20}}]\label{theorem2} Let $X$ be a smooth projective surface of degree $10$ satisfies property  ${\bf N}_{3,3}$ and $Y$ a cubic fourfold containing  $X$. Then $X$ has an $\mathcal F$-type resolution on $Y$
 $$\xymatrix{0\ar[r]&\mathcal{O}_Y^{6}(-1)\ar[r]&\mathcal{F} \ar[r]&\mathcal{I}_{X/Y}\ar[r]&0.}$$

\end{thm}

To prove Theorem \ref{theorem2}, we use matrix factorizations. Matrix factorizations were introduced by Eisenbud in his seminal paper (\cite{Eis80}. We recall here some basic facts and properties for matrix factorizations over the special case of a polynomial ring $R = K[x_0,\ldots , x_n]$, which is the case of interest for the paper. Let $Y = V(f) \subseteq \Bbb P^n$ be a hypersurface cut out by a homogeneous form $f$ of degree $s$. It is well-known that a matrix factorization $(A, B)$ of $f$ induces a maximal Cohen-Macaulay module supported on $Y$ by $\coker A$. Conversely, if we have a maximal Cohen-Macaulay $R_Y = R/(f)$-module, one has a matrix $A$ by reading off its length $1$ resolution. Indeed, it forms a part of a matrix factorization of $f$, so there is a unique matrix $B$ such that $AB = BA = f\cdot \mathrm{Id}$. As a conclusion, there is a bijection between the maximal Cohen-Macaulay modules and the equivalence classes of matrix factorizations of $f$.

 Now let us briefly recall Shamash’s construction (see \cite{Sha69}, \cite[Remark 2.6]{KiS20}). Suppose that  $X$ is a subscheme contained in hypersurface $Y=V(f)\subseteq \Bbb P^n$. Let  $R_X$ be the coordinate ring of $X$. Let $F^\bullet$ be the minimal free $R$-resolution of $R_X$. Since $X \subseteq Y$, we have a right exact sequence
$$\xymatrix{\ldots \ar[r]&F_1 \otimes_R R_Y\ar[r] &F_0 \otimes_R  R_Y\cong R_Y\ar[r]&R_X \ar[r]&0}.$$
and hence there is an $R_Y$-free resolution of $R_X$ (possibly non-minimal)
$$\xymatrix{\ldots \to  G_4\oplus G_2(-s)\oplus G_0(-2s) \to G_3\oplus G_1(-s)\to  G_2\oplus G_0(-s) \to G_1 \to G_0 \to R_X \to 0}$$
where $G_i = F_i \otimes_R R_Y$. It becomes eventually $2$-periodic, and hence induces a matrix factorization of $f$. Such a matrix factorization provides a presentation of an ACM sheaf on $Y$. 
Now, we are ready to prove Theorem \ref{theorem2}.
\begin{tpf2}
We can consider $\Gamma_*(\mathcal{O}_X)$
	as a $R_Y$-module.  Since the section ring $\Gamma_*(\mathcal{O}_X)$ have minimal free resolutions with the Betti tables as  in \ref{eqn}, the minimal resolution of $\Gamma_*(\mathcal{O}_X)$ as  a $R_Y$-module is eventually $2$-periodic with Betti numbers
		\begin{center}
		\begin{tabular}{ c | c c c c c c c c c}
			&   & $0$ & $1 $ &$2 $  &$3 $& $4 $ &$5 $  &$6 $\\ 
			\hline
			0&  & $1    $ & $\cdot$ & $\cdot$ & $\cdot$ & $\cdot$ & $\cdot$ & $\cdot$\\ 
			1&  & $\cdot$ & $\cdot    $ &$\cdot$&$\cdot$ & $\cdot$ & $\cdot$ & $\cdot$ \\ 
			2&  & $\cdot   $ & $10    $ &$  15 $  & $  6  $ & $\cdot$ & $\cdot$ & $\cdot$ \\ 
			3&  & $\cdot$ & $\cdot$ &$  \cdot  $  & $  9  $ & $  15 $ & $  6  $ & $\cdot$ \\ 
			4&  & $\cdot$ & $\cdot$ &$\cdot$  & $\cdot$ & $   \cdot $ & $  9  $ & $  15 $ \\ 

		\end{tabular}.	
	\end{center}
	Then the periodic part of its minimal free resolution
	yields, up to twist, matrix factorization of the form
	$$\xymatrix{R_Y^{15}(-3)\ar[r]^{\varphi\ \quad \quad}& R_Y^{6}(-1)\oplus R_Y^{9}(-2)\ar[r]^{\quad\quad\quad\psi}& R_Y^{15}}.$$	
	The Betti numbers of the minimal periodic resolution of $R_X$ as a $R_Y$-module
%
%
%
	is called {\it the shape of the matrix factorization of $Y$}.  Then the shape of a matrix factorization is determined by the
	Betti numbers of $\psi$. In the current case they are
	\begin{center}
		\begin{tabular}{  c c  }
			$  15 $  & $  6 $ \\ 
			$   0 $  & $  9  $ 
		\end{tabular}.
		
	\end{center}

Let $F^\bullet$ and $\overline{G^\bullet}$ be minimal free resolutions of the section ring $\Gamma_*(\mathcal{O}_X)$ as an $R$-module and $R_Y$-module respectively. Then $\varphi$ is the syzygy map $\overline{G_4} \to \overline{G_3}$. Let $\mathcal{F} = \widetilde{\coker \varphi}(-3)$.
Then the Shamash resolution starts with the Betti numbers

	\begin{center}
		\begin{tabular}{ c | c c c c c c c c c}
			&   & $0$ & $1 $ &$2 $  &$3 $\\ 
			\hline
			0&  & $1    $ & $\cdot$ & $\cdot$ & $\cdot$ \\ 
			1&  & $\cdot$ & $\cdot    $ &$1$&$\cdot$ & \\ 
			2&  & $\cdot   $ & $10    $ &$  15 $  & $  6  $\\ 
			3&  & $\cdot$ & $\cdot$ &$  \cdot  $  & $  10  $\\ 
		
		\end{tabular}
		
	\end{center}
	Thus the Shamash resolution is always non-minimal, and in a minimal resolution a cancellation occurs, causing $\beta_{1,3}=10$ to decrease by one. Such cancellation corresponds to the equation $f$ of $Y$ in $R$ becoming superfluous in $R_Y$. By definition, the map $\overline{G_3}\to  \overline{G_2}$ factors through $\mathcal{F}$. 	It follows from $F_0=R$, $F_3=R(-5)^6$ and Shamash’s construction that  we obtain a map $\alpha:\mathcal{O}_Y^{6}(-2)\to \mathcal{F}$. Thus we get the following.

\begin{claim}[{cf.\cite[Theorem 2.4]{ScT18}}]\label{cl1}
The sequence 
 $$\xymatrix{0\ar[r]&\mathcal{O}_Y^{6}(-2)\ar[r]^{\quad\alpha}&\mathcal{F} \ar[r]&\mathcal{I}_{X/Y}\ar[r]&0.}$$
 is exact.
\end{claim}
\begin{proof}
 Being the sheafification of a MCM module over $Y$, the sheaf $\mathcal F$ is a vector bundle.   We apply the functor $\Hom(\bullet, \omega_Y )$ to $\alpha$ and obtain
$ \alpha^*(-2):\mathcal{F}^*(-2)\to \mathcal{O}_Y^{6}$. The cokernel of this map coincides by construction with the cokernel of the dual of the sheafification of the last map of $F^\bullet$
$$ \mathcal{O}_Y^{15}(-1)\to \mathcal{O}_Y^{6}$$
which is a presentation of $\omega_X$ by duality on $\Bbb P^5$. Since $\rank \mathcal F=7=6+1$, both $\alpha^*(-2)$ and $\alpha$ drop rank in expected codimension $2$. Applying again $\Hom(\bullet, \omega_Y )$ to $\alpha^*(-2)$ we get that $\alpha$ is injective and  by the Hilbert–Burch Theorem, 
we have an exact complex
	$$\xymatrix{0\ar[r]&\mathcal{O}_Y^{6}(-2)\ar[r]&\mathcal{F}\ar[r]&\ar[r]\mathcal{O}_Y(\ell)&\mathcal{O}_X(\ell)\ar[r]& 0}$$
for some $\ell$. By applying again $\Hom(\bullet, \omega_Y )$ to this last exact sequence one gets that $\alpha^*(-2)$ is a presentation of $\omega_X(-\ell)$. Hence $\ell = 0$, as required.

\end{proof}
\end{tpf2}
The main result of last section is a converse to the Theorem \ref{theorem2}.  To state it, first of all let us fix our notation and terminology. Liaison has become an established technique in algebraic geometry.  The greatest activity, however, has been in the last quarter century, beginning with the work of Peskine and Szpir\'{o}  in 1974 (see \cite{PeS74}). Liaison is a powerful tool for constructing examples.

\begin{defn} Let $X_1$ and $X_2$ be equidimensional closed subschemes of dimension $r$ of $\Bbb P^n_k$. 
We say that $X_2$ is obtained by an elementary biliaison of height $m$ from $X_1$, if there exists an ACM scheme $Y \subset \Bbb P^n$, of dimension $r + 1$ containing $X_1$ and $X_2$, such that $X_2 \thicksim X_1 + mH$ on $Y$, where $H$ denotes the hyperplane class. (Here $\thicksim$ means linear equivalence of divisors on $Y$ in the sense of \cite{Har94}.) The equivalence relation generated by elementary biliaisons is called simply biliaison. If $Y$ is a complete intersection scheme in $\Bbb P^n$, we speak of CI-biliaison. If  $X_1$, $X_2$, $Y $ are contained in some projective scheme $Z \subseteq\Bbb P^n$, we speak of biliaison  (CI-biliaison) on $Z$.
\end{defn}

Casanellas and Hartshorne gave a criterion for checking when two codimension $2$ subschemes $Y_1$, $Y_2$ of a normal ACM projective scheme $X$ are in the same biliaison class (see \cite[Theorem 3.1]{CaH04}). Using such criterion, we will prove a converse to the Theorem \ref{theorem2}.
\begin{thm}\label{theorem3}
 Let  $( \psi: \mathcal{O}^6_Y(-3) \oplus
	\mathcal{O}^9_Y(-4)\to \mathcal{O}^{15}_Y(-2);\varphi : \mathcal{O}^{15}_Y(-2)  \to \mathcal{O}^6_Y \oplus
	\mathcal{O}^9_Y(-1))$ be a matrix factorization 
	 of  a cubic fourfold $Y$. Put $\mathcal F=\coker \varphi(-1)$. 
	Then, for every $a\ge 2$, there  is  an ACM smooth surface $X^\prime$ such that  $X^\prime$ has an $\mathcal F$-type resolution
$$\xymatrix{0\ar[r]&\mathcal{O}_{X^\prime}^{6}(-a)\ar[r]&\mathcal{F} \ar[r]&\mathcal{I}_{X^\prime/Y}(a^\prime)\ar[r]&0}$$
for some $a^\prime\in\Bbb Z$ and $X^\prime$ belongs to the $CI$-biliaison equivalence class of  smooth projective surfaces $X\subset Y$ of degree $10$ satisfying  property  ${\bf N}_{3,3}$.
\end{thm}
\begin{proof}

We denoted by $M=\bigoplus\limits_{n\in\Bbb Z}H^0(\mathcal F(n))$. Then $M$ is Cohen-Macaulay as $R_Y$-module and so
$\depth_R M=\depth_{R_Y}M=\dim R-1=5$.  It follows from the Auslander-Buchsbaum Formula and the shape of the matrix factorization of the cubic fourfold $Y$ that
  we have a short exact sequence
$$\xymatrix{0\ar[r]&R^{15}(-2)\ar[r]& R^{6}\oplus R^9(-1)\ar[r]& M\ar[r]&0}$$
of $R$-modules. Therefore, the Castelnuovo-Mumford regularity of $M$ is $1$.
 Let $N$ be a graded submodule of $M$ generated by homogeneous elements of $M$ all of
the same degree $a-1$. Since the Castelnuovo-Mumford regularity of $M$ is $1$, we have  $M/N$ is of finite length over $R$. We now observe that $N$ has a presentation 
$$\xymatrix{R^{c}(-a) \ar[r]^{\phi}& R^{b}(-a+1)\ar[r] &N\ar[r]&0,}$$ 
 in which the presenting matrix $\phi$ has homogeneous entries with all the entries have degree $1$.  We then adjoin indeterminates $y_{ij}$ to $R$ for $1 \le i \le b$, and $1 \le j \le 6$
to obtain a new ring $A = R[y_{ij}]$, where the $y_{ij}$'s are
considered as degree $1$ elements in the graded ring $A$. We put $M_1 = A \otimes_RM$, and $N_1 = A \otimes_RN$. We may 
consider the matrix $\phi$ as an $A$-homomorphism $A^{c}(-a) \to A^{b}(-a+1)$ presenting $N_1$, that is, we are identifying $1_A \otimes \phi$ with $\phi$. Letting $\Phi$ be the matrix with elements $y_{ij}$, we obtain an
$A$-homomorphism $A^{c}(-a) \oplus A^6(-a) \to A^{b}(-a+1)$ via the matrix $\Psi= [\phi|\Phi]$, where the vertical line represents a matrix partitioning corresponding to the two direct summands of
$A^{c}(-a) \oplus A^6(-a)$. Let $E$ be the image of $A^6(-a)$ in $N_1$ under the composition of the maps
$$\xymatrix{A^{c}(-a) \oplus A^6(-a) \ar[r]^{\quad \quad\ \Psi}& A^{b}(-a+1) \ar[r]& N_1}.$$ Then, $E$ is a free submodule of $N_1$ of rank $6$, and also $E$ is a graded submodule of $N_1$. 
Since $M$ is Cohen-Macaulay as $R_Y$-module,  $M_1$ is Cohen-Macaulay as $A$-module.
 It follows from $M_1/E$ is a graded $A$-module that its associated prime ideals are homogeneous. Since $\dim A\ge \dim R=6$, and the $(x_0,\ldots,x_5)$-depth of $M$ over $R$ is equal to the $(x_0,\ldots,x_5)$-depth of $M_1$ over $A$, we get that $\depth_{(x_0,\ldots,x_5)}(M_1/E)=4$. Hence, because our field is infinite, there must be a form $g$ of degree one in $x_0, ...,x_5$ that is not a zero divisor on $M_1/E$. 

Next we show that $A_g \otimes_A(N_1/E)$ is isomorphic to a normal prime ideal $I_g$ of $A_g$. Indeed,  since  degree of $g$ is one in $x_0,\ldots, x_5$, we have that $M_g$ at $N_g$ is a free $A_g$-module. Hence, $(N_1)_g = A_g \otimes_A N_g$ is also a free $A_g$-module. Thus the sequence
$$\xymatrix{A_g^{c}(-a) \ar[r]^{\phi}& A_g^{b}(-a+1)\ar[r] &(N_1)_g\ar[r]&0}$$ 
splits in which the presenting matrix $\phi$ is as a matrix of homogeneous entries with all the entries belong to $A_g$. After performing invertible row and column operations on $\phi$, we obtain a $b\times c$- matrix of the form
\begin{center}\begin{tabular}{ |c | c| }
\hline
$\bf 0$ &$\bf{I}_t$\\
\hline
$\bf 0$&$\bf0$\\
\hline
\end{tabular}
\end{center}
where  $\bf I_t$ is the $t$ by $t$ identity matrix for some $t$. The effect of the row operations on the matrix 
$\Psi= [\phi|\Phi]$, is to obtain a matrix $\Psi^\prime= [\phi^\prime|\Phi^\prime]$, where $\Phi^\prime$ consists of entries $y^\prime_{ij}$ which are a new set of $6b$ polynomial indeterminates over $R_g$. After applying further invertible column operations on $\Psi$, we obtain the 
\begin{center}$\Psi^\prime$=\begin{tabular}{ |c | c  | c|}
\hline
$\bf 0$ &$\bf{I}_t$&$\bf 0$\\
\hline
$\bf 0$&$\bf0$&$y^{\prime\prime}_{ij}$\\
\hline
\end{tabular}
\end{center}
where the $y^{\prime\prime}_{ij}$ in the lower right-hand corner is a subset of the $y^{\prime}_{ij}$. Let $I_g$ be the $A_g$-ideal  of projective dimension one,  generated by the maximal minors of the matrix $(y^{\prime\prime}_{ij})$ in the lower right-hand corner of $\Psi^\prime$. Then, $A_g \otimes_A(N_1/E)\cong I_g$ is a normal prime ideal of $A_g$.

Since $g$ is regular on $M_1/E$ and $A_g \otimes_A(N_1/E)\cong I_g$, there exists a graded ideal $I$ of $A$ such that $M_1/E\cong I$ and $I_g=A_g\otimes I$. Since $A$ is factorial and $M_1$ is a non free reflexive module we may assume that $I$ has height exactly two. Since this matrix consists of indeterminates over $R_g$, it follows that $I_g$ is a normal prime ideal of $A_g$, and so $I$ a normal prime ideal of $A$. Therefore, we have an exact sequence
$$\xymatrix{0\ar[r] &A(-a)^6\ar[r] & M_1\ar[r]&I\ar[r]&0,}$$
in which $I$ is a homogeneous prime ideal of $A$ of height two, and $A/I$ is Cohen-Macaulay. Then, it follows from Flenner's version of Bertini's Theorem that there exists a  linear regular sequence $h_1, \ldots,h_{6b}$ on $A/I$. Hence  the sequence 
$$\xymatrix{0\ar[r] &(A/(h_1,\ldots,h_{6b})A)(-a)^6\ar[r] & M_1/(h_1,\ldots,h_{6b})M_1\ar[r]&I/(h_1,\ldots,h_{6b})I\ar[r]&0}$$
is exact. Since $M_1/(h_1,\ldots,h_{6b})M_1\cong M$ and $A/(h_1,\ldots,h_{6b})A\cong R$,  there  is  a codimension $2$ subscheme  $X^\prime \subset Y$ without embedded components such that  $X^\prime$ has an $\mathcal F$-type resolution
 $$\xymatrix{0\ar[r]&\mathcal{O}_Y^{6}(-a)\ar[r]^{\quad\alpha}&\mathcal{F} \ar[r]&\mathcal{I}_{X^\prime/Y}(a^\prime)\ar[r]&0.}$$
for some $a^\prime\in\Bbb Z$. By Theorem 3.1 in \cite{CaH04} and Theorem \ref{theorem2}, $X^\prime$ and  smooth projective surfaces $X\subset Y$ of degree $10$ satisfying  property  ${\bf N}_{3,3}$ are in the $CI$-biliaison equivalence class on $Y$. By theorem \ref{theorem1}, $X$ is a projective Cohen-Macaulay smooth surface. Therefore $X^\prime$ is a projective Cohen-Macaulay smooth surface, as required.

\end{proof}
Here is an immediate corollary of Theorem \ref{theorem2} and \ref{theorem3} as follows:
\begin{cor}
	Let $Y$ be a smooth cubic fourfold. Then the statements are equivalent. 
\begin{itemize}
\item[$1)$]  $Y\in \mathcal{C}_\delta$ for $6<\delta\le 44$ and $\delta\equiv 2 \pmod 6$.
\item[$2)$]	 $Y$ admits the shape of the matrix factorization
\begin{center}
	\begin{tabular}{  c c c c }
		$  15 $  & $  6  $ & $\cdot$  \\ 
		$\cdot$  & $  9  $ & $  15 $\\ 
		
	\end{tabular}	
\end{center}

\end{itemize}
	
\end{cor}

\medskip

{\bf\large Acknowledgement} The author wishes to thank Prof. F.-O. Schreyer for his indispensable suggestions about the topics in this paper. 



\end{document}